\documentclass[graybox]{svmult}

\usepackage{type1cm}
\usepackage{makeidx}
\usepackage{graphicx}
\usepackage{multicol}
\usepackage[bottom]{footmisc}

\usepackage{newtxtext}
\usepackage[varvw]{newtxmath}

\makeindex

\begin{document}

\title{Some properties of Besov-Morrey type spaces}

\author{
Nazerke Tleukhanova\inst{1}\inst{2}
\and Kelbet Sadykova\inst{3}
\and Aigerim Ussentay\inst{1}
}

\institute{
\inst{1} L.N. Gumilyov Eurasian National University, Astana, Kazakhstan \\
\inst{2} Geometry LLP, Astana, Kazakhstan \\
\inst{3} S. Seifullin Kazakh Agro-Technical Research University, Astana, Kazakhstan
}

\maketitle

\begin{abstract}
 This paper considers Sobolev-Morrey and Besov–Morrey spaces. For Morrey spaces, multipliers are studied, and a theorem on multipliers from $M_p^{\lambda}(\mathbb{T})$ to $M_p^{\lambda}(\mathbb{T})$ for \,\,\,$1<p<\infty,0\le\lambda<\frac{1}{p}$ is obtained. Based on this theorem, embeddings of $B_1^\alpha\left(M_p^\lambda\right)\hookrightarrow W^\alpha\left(M_p^\lambda\right)\hookrightarrow B_\infty^\alpha\left(M_p^\lambda\right)$ for \,\,\,$0<p\le\infty$, $0\le\lambda\le\frac{1}{p}$, $\alpha\in\mathbb R$ are derived. An interpolation theorem for Besov–Morrey spaces is proved.
\end{abstract}

\section{Introduction}

The theory of function spaces constitutes one of the most important foundations for the analysis of functions, distributions, and solutions to partial differential equations (PDEs). Classical scales such as Sobolev spaces, Besov spaces, and Triebel–Lizorkin spaces provide flexible tools for studying the interplay between smoothness and integrability.

However, in many PDE problems it is not sufficient to measure norms in a purely global sense. Instead, a more delicate control of the local behavior of functions is required, especially when the data or solutions exhibit nonuniform structures or singularities. In such situations, spaces of local integrability, in particular Morrey spaces, play a crucial role.

By combining ideas of smoothness and frequency analysis from Besov and Triebel–Lizorkin spaces with the concepts of local integrability and scale sensitivity inherent in Morrey spaces, new function spaces were introduced that simultaneously control smoothness and local integrability. These spaces are well suited to capture inhomogeneities, local irregularities, singularities, and distributional data.

This development significantly broadened the range of analytical tools available for the study of nonlinear PDEs, especially in settings where classical Sobolev or Besov spaces are no longer adequate.

The first serious and systematic step in this direction was made in the seminal work of Hideo Kozono and Masao Yamazaki \cite{KozYam}, where Besov-type versions of Morrey spaces were introduced.

A major milestone was the appearance of the monograph \cite{YuSiYa}, which presented a unified framework combining Besov and Triebel–Lizorkin spaces with Morrey and Campanato spaces.

An important continuation of this line of research was provided by the fundamental study of Besov–Morrey spaces in \cite{Mazzucato}. Further developments and refinements can be found in \cite{IzNoi}, \cite{Sickel}, \cite{HarSkr}, \cite{Sai}, and \cite{OkaZhan}.

For the analysis of trigonometric series and integral operators in these scales, the classical works on harmonic analysis in \cite{Zigmu} and \cite{Grafakos} are utilized. The theory of interpolation of operators, which is critical for establishing interpolation identities, relies on the works \cite{Bennet}, as well as on modern results in \cite{B. Ch. N. } regarding Marcinkiewicz-type theorems for Morrey-type spaces. 

The investigation of convolution properties and O'Neil-type inequalities in various function spaces (Lorentz, Besov, Triebel–Lizorkin) is presented in a series of works \cite{NurTich}, \cite{ST-EMJ}, \cite{ST-KSU}, \cite{ST-KazNU}, \cite{ST-KazNPU}, \cite{BekKerNur}. Recent research also encompasses convolution operators directly in Morrey spaces (\cite{ND}), integral operators on the Heisenberg group (\cite{Gul}), and O'Neil-type inequalities for Sobolev-Morrey spaces (\cite{ISAAC}).

\section{The Besov-Morrey spaces}

Let $0\le\lambda\le\frac{n}{p}$ and $0<p<\infty$. All set of functions $f{\in{L^{loc}_{p}}}$ called Morrey spaces, if 
$$||f||_{M^{\lambda}_p}=\sup_{x\in \mathbb R^n}\sup_{r>0} r^{-\lambda}||f||_{L_p(B(x,r))}<\infty$$

 where, $B_r(x)$ is the ball centered at point $x$ and with radius $r>0$.  
If we notice that $\lambda=0 $, then $M_{p}^{0}(\mathbb R^n)=L_p(\mathbb R^n)$, indeed, let $f\in M_{p}^{0}$, then  
$$||f||_{M^{0}_p}=\sup_{x\in \mathbb R^n}\sup_{r>0}||f||_{L_p(B(x,r))}=||f||_{L_p(\mathbb R^n)}.$$

If \, $\lambda=\frac{n}{p}$  and  $0<p<\infty$, then it holds    $M_{p}^{\frac{n}{p}}(\mathbb R^n)=L_{\infty}(\mathbb R^n)$.

 If\, $\lambda<0$ \,or\, $\lambda>\frac n{p}$, then $M_p^{\lambda}=\Theta$, where, $\Theta$  is the set of all functions equivalent to zero on $\mathbb R^n$. 

Let \;$\mathbb{T}=[0,1)$ be $one$-dimensional torus, and  let $\mu$ denote the $one$-dimensional Lebesgue measure, $1<p<\infty$, \,$0<q\le\infty$\, $0\le\lambda\le\frac{n}{p}$,\, $\alpha\in \mathbb R$.

We will denote by \,$W^{\alpha}(M_p^{\lambda})$ \,the set of all trigonometric series $f=\displaystyle\sum_{m\in \mathbb Z}a_me^{2\pi imx}$ (generally speaking, divergent), such that there exists a function $f^\alpha\in M_p^{\lambda}$ with a Fourier series $f^\alpha\sim\displaystyle\sum_{m\in \mathbb Z}\bar {m}^{\alpha} a_m e^{2\pi imx}$. This definition allows us to consider spaces \, $W^{\alpha}(M_p^{\lambda})$ \,for all \,$\lambda\in \mathbb R$ \,without resorting to distribution theory. 

And its norm is defined by 
   $$ \|f\|_{W^{\alpha}(M_{p}^{\lambda})}=\|f^\alpha\|_{M_{p}^{\lambda}}.$$

The following chain of continuous embeddings holds:
$$W_q^{\alpha}\hookrightarrow W^{\alpha}(M_p^{\lambda})\hookrightarrow W_p^{\alpha}.$$ 
Now let us define the Besov–Morrey space. 

Let
\begin{equation}
       Q_k=[-2^k,2^k]\cap\mathbb Z 
\end{equation}
is a segment centered at the point $O$, with edge $2^k$.

Let $f\in L_1(\mathbb T)$,
    $$f\sim\sum_{m\in\ \mathbb{Z}}\hat f_m e^{2\pi imx},$$

Let’s define  
\begin{equation}
    \Delta_kf=\sum_{m\in Q_k\backslash{Q_{k-1}}} \hat f_me^{2\pi imx}, k=0,1...,
\end{equation}
here at $k=0$, $Q_{k-1}=\oslash$.

Let \,$0<p,q\leq\infty$, \,$0\le\lambda\le\frac{1}{p}$ and $\beta\in \mathbb R$.

Let us define the Besov-Morrey space $B_{q}^\beta\left(M_p^\lambda\right)$ as the set the set of all trigonometric series $f=\displaystyle\sum_{m\in \mathbb Z}a_me^{2\pi imx}$ (generally speaking, divergent) from \,$L_1(\mathbb T)$, \,such that 
\begin{equation}
 \|f\|_{B_{q}^\beta\left(M_p^\lambda\right)}:=\left(\sum_{k=0}^{\infty}\left(2^{\beta k}\|\Delta_kf\|_{M^\lambda_{p}}\right)^q\right)^\frac{1}{q}<\infty,  \quad \text{where}\,   \,\,1< q<\infty
\end{equation}
and in case $q=\infty$
$$ \|f\|_{B_{\infty}^\beta\left(M_p^\lambda\right)}:=\sup_{k}2^{\beta k}\|\Delta_kf\|_{M^\lambda_{p}}<\infty.$$

The following chain of continuous embeddings holds:
$$B_{q,\tau}^{\alpha}\hookrightarrow B_{\tau}^{\alpha}(M_p^{\lambda})\hookrightarrow B_{p,\tau}^{\alpha}.$$

\section{Multipliers in Morrey spaces}
\label{sec:2}
\begin{theorem}\label{theoremMarsel-Riesz}
Let $1<p<\infty$, $0\le \lambda<\frac{1}{p}$,\,  then  the partial sum operator of the Fourier series $ S_n\left(f,y\right)=\int_{-\pi}^{\pi}f(x)D_n(y-x)dx$ \,is bounded from the Morrey space ${M_p^{\lambda}}((\mathbb T)$\,into \,${M_p^{\lambda}}(\mathbb T)$.
\end{theorem}
\begin{proof}
For the partial sum operator of a Fourier series, the following representation holds:
\begin{equation*}
\begin{split}
 S_n\left(f,y\right)=\int_{-\pi}^{\pi}f(x)D_n(y-x)dx=\\
 =\int_{-\pi}^{\pi}f(x)\frac{\sin(\frac1{2}+n)(y-x)}{2sin(\frac{y-x}{2})}dx.
\end{split}
\end{equation*}
This is operator is bounded from $L_p(\mathbb T)$ to $L_p(\mathbb T)$ \cite {Zigmu}.

On the other hand 
\begin {equation*}
\begin{split}
|S_n(f,y)|\le c \int_{-\pi}^{\pi}\frac{|f(x)|}{|y-x|}dx, \,\,\,\,\,\,  {x}\notin \text{supp} f. 
\end{split}
\end{equation*}
Then from theorem 5.1 \cite {B. Ch. N. }  \,\,the operator $S_n(f,x)$ bounded from \,\,${M_p^{\lambda}}(\mathbb T)$ \,\,to\,\, ${M_p^{\lambda}}(\mathbb T)$.
\end{proof}

\begin{definition}
Let\,\, $1<p\le q<\infty$. We say that a sequence of complex numbers 
$\gamma=\{\gamma_k\}_{k\in \mathbb Z}$ is a multiplier of trigonometric Fourier series from $M_p^\lambda(\mathbb T)$ to $M_q^\lambda(\mathbb T)$\,\,,if, for an arbitrary function $f\in M_p^\lambda(\mathbb T)$ with Fourier series $\displaystyle\sum_{k\in\mathbb Z}\hat f_k e^{2\pi{ikx}}$, there exists a function $f_{\gamma}$ from $M_q^\lambda(\mathbb T)$
  whose Fourier series coincides with the series $\displaystyle\sum_{k\in\mathbb Z}\gamma_k\hat f_k e^{2\pi ikx}$, and the operator $\,\,T_{\gamma}f=f_{\gamma}$ is a bounded operator from $M_p^\lambda(\mathbb T)$ to $M_q^\lambda(\mathbb T)$. The set $m_p^q$- of all multipliers defined in this way forms a linear space with a norm
  $$\|\gamma\|_{m_p^q}=\sup_{f\neq0}\frac{{\|f_\gamma\|}_{M_q^{\lambda}}}{{\|f\|_{M_p^{\lambda}}}}$$
\end{definition}

We will prove the theorem on Fourier series multipliers in Morrey spaces.
\begin{theorem}\label{theoremMultipl}
Let $1<p<\infty$, $0\le\lambda<\frac{1}{p}$. If $\displaystyle\sum_{k=-\infty}^{\infty}|\Delta_k\gamma|<\infty$, where $\Delta_k\gamma=\gamma_k-\gamma_{k+1}$, then $\gamma=\{\gamma_k\}_{k=-\infty}^{\infty}$ is a multiplier from ${M_p^{\lambda}}(\mathbb T)$ to ${M_p^{\lambda}}(\mathbb T)$. 
\end{theorem}
\begin{proof}
Without loss of generality, let us set $f=\displaystyle\sum_{k=0}^{\infty}\hat f_ke^{2\pi ikx}$. Then 
\begin {equation*}
\begin{split}
\|f^{\alpha}\|_{M_p^{\lambda}(\mathbb T)}=\left\|\displaystyle\sum_{k=0}^{\infty}\gamma_k \hat f_k e^{2\pi ikx}\right\|_{M_p^{\lambda}(\mathbb T)}=\left\|\displaystyle\sum_{k=0}^{\infty}\displaystyle\sum_{m=k}^{\infty}\left(\gamma_m-\gamma_{m-1}\right)\hat f_k e^{2\pi ikx}\right\|_{M_p^{\lambda}(\mathbb T)}=\\
=\left\|\displaystyle\sum_{m=0}^{\infty}\left(\gamma_m-\gamma_{m+1}\right)\displaystyle\sum_{k=0}^{m}\hat f_k e^{2\pi ikx}\right\|_{M_p^{\lambda}(\mathbb T)}\le\displaystyle\sum_{m=0}^{\infty}|\gamma_m-\gamma_{m+1}| \|S_n(f, \cdot)\|_{M_p^{\lambda}(\mathbb T)}.
\end{split}
\end{equation*}
Applying theorem\,\,\ref{theoremMarsel-Riesz}, we obtain 
\begin {equation*}
\begin{split}
\|f^{\alpha}\|_{M_p^{\lambda}(\mathbb T)}\le\displaystyle\sum_{m=1}^{\infty}|\gamma_m-\gamma_{m+1}| \|f\|_{M_p^{\lambda}(\mathbb T)} 
\end{split}
\end{equation*}
\end{proof}

\begin{theorem}\label{theoremVlozh}
Let \,\,\,$0<p\le\infty$, $0\le\lambda\le\frac{1}{p}$, $\alpha\in\mathbb R$. Then 
$$B_1^\alpha\left(M_p^\lambda\right)\hookrightarrow W^\alpha\left(M_p^\lambda\right)\hookrightarrow B_\infty^\alpha\left(M_p^\lambda\right).$$
\end{theorem}
\begin{proof} 
We will show the embedding 
\begin{equation*}
B_1^\alpha\left(M_p^\lambda\right)\hookrightarrow W^\alpha\left(M_p^\lambda\right).
\end{equation*}
Let $f\in B_1^{\alpha}(M_p^{\lambda})$.\,\,\,Its trigonometric series is given by $f=\displaystyle\sum_{m\in \mathbb{Z}}a_me^{2 \pi im x}$.

By using the relations $$\Delta_kf_\gamma^\alpha=2^{k\alpha}\displaystyle\sum_{m=2^k}^{2^{k+1}}\frac{{\bar m}^\alpha}{2^{k\alpha}}a_m e^{2\pi imx}$$
and 
\begin{equation*}
\begin{split}
\frac{1}{2^{k\alpha}}\displaystyle\sum_{m=2^k}^{2^{k+1}}|m^\alpha-(m+1)^\alpha|\asymp\frac1{2^{k\alpha}}\displaystyle\sum_{m=2^k}^{2^{k+1}}m^{\alpha-1}=\\
=\frac1{2^{k\alpha}}\left(2^{(k+1)\alpha}-2^{k\alpha}\right)\thickapprox 1
\end{split}
\end{equation*}
Using the estimate $\|\Delta_k f^\alpha\|\le2^{k\alpha}\|\Delta_k f\|$ and we apply theorem \,\,\ref{theoremMultipl}  on multipliers
\begin{equation*}
\begin{split}
\|f\|_{W^{\alpha}(M_p^\lambda)}=\|f^\alpha\|_{M_p^\lambda}=\|\displaystyle\sum_{k=0}^{\infty}\Delta_k f^\alpha\|_{M_p^\lambda}\le\\
\le\displaystyle\sum_{k=0}^{\infty}\|\Delta_k f^\alpha\|_{M_p^\lambda}\le \displaystyle\sum_{k=0}^{\infty}2^{k\alpha}\|\Delta_k f\|_{M_p^\lambda}=\|f\|_{B_1^\alpha(M_p^{\lambda})},
\end{split}
\end{equation*}

We now prove the embedding 
$$W^{\alpha}(M_p^{\lambda})\hookrightarrow B_{\infty}^{\alpha}(M_p^{\lambda}).$$
Let $f\in W^{\alpha}(M_p^{\lambda})$,\,\,\,Its trigonometric series is given by $f=\sum_{m\in \mathbb{Z}}a_me^{2 \pi im x}$.
 
 Using the relations below
\begin{equation*}
\begin{split}
2^{k\alpha}|\Delta_k f|=2^{k\alpha}\displaystyle\sum_{m=2^k}^{2^{k+1}}\hat f_m e^{2\pi imx}=\\
=\displaystyle\sum_{m=2^k}^{2^{k+1}}\frac{2^{k\alpha}}{\bar{m}^{\alpha}}\bar{m}^{\alpha}a_me^{2\pi imx},
\end{split}
\end{equation*}
 
Further, we have 
\begin{equation*}
\begin{split}
2^{k\alpha}\displaystyle\sum_{m=2^k}^{2^{k+1}}|\frac1{m^\alpha}-\frac1{(m+1)^\alpha}|\asymp 2^{k\alpha}\displaystyle\sum_{m=2^k}^{2^{k+1}}m^{-1-\alpha}\asymp\\
\asymp2^{k\alpha}(-2^{-(k+1)\alpha}+2^{-k\alpha})\asymp1.
\end{split}
\end{equation*}
According to Theorem \,\,\ref{theoremMultipl} on multipliers,  where \,\,$\left\{\frac{2^{k\alpha}}{\bar{m}^{\alpha}}\right\}_{m=2^k}^{2^{k+1}}$ is a multiplier from $M_p^{\lambda}(\mathbb T)$ to $M_p^{\lambda}(\mathbb T)$, therefore we have 
\begin{equation*}
\begin{split}
\|f\|_{B_{\infty}^{\alpha}(M_p^\lambda)}=\sup_{k\in \mathbb Z_+}2^{-k\alpha}\|\Delta_k f\|_{M_p^\lambda}\le\\
\le\sup_{k\in \mathbb Z_+}\|\Delta_k f^\alpha\|_{M_p^\lambda}\lesssim \|f^\alpha\|_{M_p^\lambda}=\\
=\|f\|_{W^\alpha(M_p^\lambda)}. 
\end{split}
\end{equation*}
We have also applied Theorem \,\,\ref{theoremMarsel-Riesz}. 
\end{proof}

\begin{theorem}\label{thInterpol}
Let $0<\beta_1<\beta_0<\infty$, $0<q_0,q_1,h\le\infty$, \,\,$0<p\le\infty$,\,\,$0\le\lambda\le\frac1{p}$ and
$$\beta=(1-\theta)\beta_0+\theta\beta_1.$$
We have
$$\left(B_{q_0}^{\beta_0}\left(M_p^{\lambda}\right),B_{q_1}^{\beta_1}\left(M_p^{\lambda}\right)\right)_{\theta,h}=B_{h}^{\beta}\left(M_p^{\lambda}\right).$$
\end{theorem}

\begin{proof}[\bf \indent Proof]
Let $f\in \left(B_{q_0}^{\beta_0}\left(M_p^{\lambda}\right),B_{q_1}^{\beta_1}\left(M_p^{\lambda}\right)\right)_{\theta,h}$. Take any representation $f=f_0+f_1$ such that $f_0\in B_{q_0}^{\beta_0}\left(M_p^{\lambda}\right)$ and $f_1\in B_{q_1}^{\beta_1}\left(M_p^{\lambda}\right)$. Then for any $k\in\mathbb Z_+$ we get

\begin{equation*}
\begin{split}
2^{\beta k}\|\Delta_kf\|_{M_p^\lambda}&\le2^{(\beta-\beta_0)k}\left(2^{\beta_0k}\|\Delta_kf\|_{M_p^\lambda}+2^{(\beta_0-\beta_1)k}2^{\beta_1k}\|\Delta_kf\|_{M_p^\lambda}\right)\\
&\le2^{(\beta-\beta_0)k}\left(\|f\|_{B_\infty^{\beta_0}(M_p^\lambda)}+2^{(\beta_0-\beta_1)k}\|f\|_{B_\infty^{\beta_1}(M_p^{\lambda})}\right).
\end{split}
\end{equation*}
Using $$B_\infty^{\beta}(M_p^{\lambda})\hookleftarrow B_q^{\beta}(M_p^{\lambda}), \forall 0<q\le\infty$$
$$2^{\beta k}\|\Delta_kf\|_{M_p^\lambda}\le 2^{(\beta-\beta_0)k}\left(\|f\|_{B_{q_0}^{\beta_0}(M_p^{\lambda})}+2^{(\beta_0-\beta_1)k}\|f\|_{B_{q_1}^{\beta_1}(M_p^{\lambda})}\right)$$
Taking into account our choice of $f_0$ and $f_1$, by definition of the $K$-functional $K(f,t):=K\left(f,t;B_{q_0}^{\beta_0}\left(M_p^{\lambda}\right),B_{q_1}^{\beta_1}\left(M_p^{\lambda}\right)\right)$, we have
$$ 2^{\beta k}\|\Delta_kf\|_{M_p^\lambda}\lesssim 2^{(\beta-\beta_0)k}K\left(f,2^{(\beta_0-\beta_1)k}\right)=2^{-\theta(\beta_0-\beta_1)k}K\left(f,2^{(\beta_0-\beta_1)k}\right).$$
Therefore,
\begin{equation*}
\begin{split}
\|f\|_{B_{h}^\beta\left(M_p^\lambda\right)}&\lesssim \left(\sum_{k=0}^\infty\left(2^{-\theta(\beta_0-\beta_1)k}K\left(f,2^{(\beta_0-\beta_1)k}\right)\right)^h\right)^{\frac1h}\\
&\asymp\left(\int\limits_{1}^\infty\left(t^{-\theta(\beta_0-\beta_1)}K\left(f,t^{\beta_0-\beta_1}\right)\right)^h\frac{dt}t\right)^{\frac1h}\\
&\lesssim \left(\int\limits_{1}^\infty\left(t^{-\theta}K\left(f,t\right)\right)^h\frac{dt}t\right)^{\frac1h}\le \|f\|_{\left(B_{q_0}^{\beta_0}\left(M_p^{\lambda}\right),B_{q_1}^{\beta_1}\left(M_p^{\lambda}\right)\right)_{\theta,h}},
\end{split}
\end{equation*}
i.e.,
$$\left(B_{q_0}^{\beta_0}\left(M_p^{\lambda}\right),B_{q_1}^{\beta_1}\left(M_p^{\lambda}\right)\right)_{\theta,h}\hookrightarrow B_{h}^{\beta}\left(M_p^{\lambda}\right).$$
Let us show the inverse embedding. Consider now
\begin{equation*}
\begin{split}
\|f\|_{\left(B_{q_0}^{\beta_0}\left(M_p^{\lambda}\right),B_{q_1}^{\beta_1}\left(M_p^{\lambda}\right)\right)_{\theta,h}}&=\left(\int\limits_0^\infty\left(t^{\theta}K(f,t)\right)^h\frac{dt}{t}\right)^{\frac1h}\\
&=\frac1{(\beta_0-\beta_1)^{\frac1h}}\left(\int\limits_0^\infty\left(t^{-\theta(\beta_0-\beta_1)}K\left(f,t^{\beta_0-\beta_1}\right)\right)^h\frac{dt}t\right)^{\frac1h}.
\end{split}
\end{equation*}
Since 
$$\displaystyle K\left(f,t^{\beta_0-\beta_1}\right)=\inf_{f=f_0+f_1}\left(\|f_0\|_{B_{q_0}^{\beta_0}\left(M_p^\lambda\right)}+t^{\beta_0-\beta_1} \|f_1\|_{B_{q_1}^{\beta_1}\left(M_p^\lambda\right)} \right)\le t^{\beta_0-\beta_1}\|f\|_{B_{q_1}^{\beta_1}\left(M_p^\lambda\right)},$$
have
\begin{equation*}
\begin{split}
\|f\|_{\left(B_{q_0}^{\beta_0}\left(M_p^{\lambda}\right),B_{q_1}^{\beta_1}\left(M_p^{\lambda}\right)\right)_{\theta,h}}&\lesssim\left[\int\limits_0^1\left(t^{-\theta(\beta_0-\beta_1)}t^{\beta_0-\beta_1}\|f\|_{B_{q_1}^{\beta_1}\left(M_p^\lambda\right)}\right)^h\frac{dt}t\right.\\
&+\left.\int\limits_1^\infty\left(t^{-\theta(\beta_0-\beta_1)}K\left(f,t^{\beta_0-\beta_1}\right)\right)^h\frac{dt}t\right]^{\frac1h}\\
&\asymp\|f\|_{B_{q_1}^{\beta_1}\left(M_p^\lambda\right)}+\left(\sum_{k=0}^\infty\left(2^{-\theta(\beta_0-\beta_1)k}K\left(f,2^{(\beta_0-\beta_1)k}\right)\right)^h\right)^{\frac1h}.
\end{split}
\end{equation*}
In view of $\beta_1<\beta$, we get $\|f\|_{B_{q_1}^{\beta_1}\left(M_p^\lambda\right)}\lesssim \|f\|_{B_{h}^{\beta}\left(M_p^\lambda\right)}$. 

Let $f\in B_{h}^{\beta}\left(M_p^{\lambda}\right)$, $\tau=\min\left(q_0,q_1,h\right)$, and $\beta\in\mathbb Z^+$. Define $f_0$ and $f_1$ as follows:
$$f_0(x):=S_{2^l}f(x)=\sum_{k=-2^l}^{2^l}a_ke^{2\pi ikx}$$
and
$$f_1(x):=f(x)-f_0(x)=\sum_{|k|>2^l}a_k e^{2\pi ikx}.$$
Then by Jensen’s inequality
\begin{equation*}
\begin{split}
\|f_0\|_{B_{q_0}^{\beta_0}\left(M_p^\lambda\right)}&=\left(\sum_{k=0}^\infty\left(2^{\beta_0k}\|\Delta_k f_0\|_{M_p^\lambda}\right)^{q_0}\right)^{\frac1{q_0}}\le \left(\sum_{k=0}^\infty\left(2^{\beta_0k}\|\Delta_k f_0\|_{M_p^\lambda}\right)^{\tau}\right)^{\frac1{\tau}}\\
&= \left(\sum_{k=0}^l\left(2^{\beta_0k}\|\Delta_k f\|_{M_p^\lambda}\right)^{\tau}\right)^{\frac1{\tau}}
\end{split}
\end{equation*}
and
\begin{equation*}
\begin{split}
\|f_1\|_{B_{q_1}^{\beta_1}\left(M_p^\lambda\right)}&\lesssim 2^{\beta_1l}\|\Delta_k f_1\|_{M_p^\lambda}+\left(\sum_{k=l+1}^\infty\left(2^{\beta_1k}\|\Delta_k f_1\|_{M_p^\lambda}\right)^{\tau}\right)^{\frac1{\tau}}\\
&\lesssim \left(\sum_{k=0}^l\left(2^{\beta_1k}\|\Delta_k f_1\|_{M_p^\lambda}\right)^{\tau}\right)^{\frac1{\tau}}+\left(\sum_{k=l+1}^\infty\left(2^{\beta_1k}\|\Delta_k f_1\|_{M_p^\lambda}\right)^{\tau}\right)^{\frac1{\tau}}\\
&=\left(\sum_{k=l+1}^\infty\left(2^{\beta_1k}\|\Delta_k f\|_{M_p^\lambda}\right)^{\tau}\right)^{\frac1{\tau}}.
\end{split}
\end{equation*}
 Then, using
$$K\left(f,2^{(\beta_0-\beta_1)k}\right)\le\|f_0\|_{B_{q_0}^{\beta_0}\left(M_p^{\lambda}\right)}+2^{(\beta_0-\beta_1)k}\|f_1\|_{B_{q_1}^{\beta_1}\left(M_p^{\lambda}\right)}$$
and the above estimates, we have
\begin{equation*}
\begin{split}
\|f\|_{\left(B_{q_0}^{\beta_0}\left(M_p^{\lambda}\right),B_{q_1}^{\beta_1}\left(M_p^{\lambda}\right)\right)_{\theta,h}}&\lesssim\|f\|_{B_{h}^{\beta}\left(M_p^{\lambda}\right)}+\left(\sum_{l=0}^\infty2^{-\theta(\beta_0-\beta_1)hl}\left\{\left[\sum_{k=0}^l\left(2^{\beta_0k}\|f\|_{M_p^\lambda}\right)^\tau\right]^{\frac1\tau}\right.\right.\\
&\left.\left.+2^{(\beta_0-\beta_1)l}\left[\sum_{k=l+1}^\infty\left(2^{\beta_1k}\|f\|_{M_p^\lambda}\right)^\tau\right]^{\frac1\tau}\right\}^h\right)^{\frac1h}.
\end{split}
\end{equation*}
Further, since $\tau\le h$ we apply Hardy’s inequality  \cite{ND}  to get
$$\|f\|_{\left(B_{q_0}^{\beta_0}\left(M_p^{\lambda}\right),B_{q_1}^{\beta_1}\left(M_p^{\lambda}\right)\right)_{\theta,h}}\lesssim\|f\|_{B_{h}^{\beta}\left(M_p^{\lambda}\right)}.$$
This completes the proof.
\end{proof}

\begin{theorem}
Let $0<\beta_1<\beta_0<\infty$, $0<p,q\le\infty$, $1\le\lambda\le\frac1{p}$  and
$$\beta=(1-\theta)\beta_0+\theta\beta_1.$$
Then the following interpolation identity holds
$$\left(W^{\beta_0}\left(M_p^\lambda\right),W^{\beta_1}\left(M_p^\lambda\right)\right)_{\theta,q}=B_q^\beta\left(M_p^\lambda\right).$$
\end{theorem}

\begin{proof}
By theorem \ref{theoremVlozh}, we have the continuous embeddings
$$B_1^{\beta_0}\left(M_p^\lambda\right)\hookrightarrow W^{\beta_0}\left(M_p^\lambda\right)\hookrightarrow B_\infty^{\beta_0}\left(M_p^\lambda\right)$$
and similarly,
$$B_1^{\beta_1}\left(M_p^\lambda\right)\hookrightarrow W^{\beta_1}\left(M_p^\lambda\right)\hookrightarrow B_\infty^{\beta_1}\left(M_p^\lambda\right).$$
It follows that
$$\left(B_1^{\beta_0}(M_p^\lambda),B_1^{\beta_0}(M_p^\lambda)\right)_{\theta,q}\hookrightarrow \left(W^{\beta_0}(M_p^\lambda),W^{\beta_0}(M_p^\lambda)\right)_{\theta,q}\hookrightarrow \left(B_\infty^{\beta_0}(M_p^\lambda),B_\infty^{\beta_0}(M_p^\lambda)\right)_{\theta,q}.$$
By applying Theorem \ref{thInterpol}, we obtain
$$B_q^{\beta}\left(M_p^\lambda\right)\hookrightarrow \left(W^{\beta_0}\left(M_p^\lambda\right),W^{\beta_0}\left(M_p^\lambda\right)\right)_{\theta,q}\hookrightarrow B_q^{\beta}\left(M_p^\lambda\right).$$
Therefore, the interpolation space coincides with the Besov space
$$\left(W^{\beta_0}\left(M_p^\lambda\right),W^{\beta_1}\left(M_p^\lambda\right)\right)_{\theta,q}=B_q^\beta\left(M_p^\lambda\right).$$
\end{proof}

\end{document}